\documentclass[a4paper,11pt]{amsart}

\pdfoutput=1

\usepackage[T1]{fontenc}
\usepackage{amssymb,amsfonts,amsmath,amsthm}
\usepackage{url}
\usepackage{color}
\usepackage{enumerate}
\usepackage{epsfig,graphicx,psfrag}
\usepackage{tikz}
\usepackage{asymptote}
\usepackage{pinlabel}

\usepackage[top=1.2in,bottom=1.4in,left=1.4in,right=1.4in]{geometry}

\input{xy}
\xyoption{all}
\objectmargin={3mm}

\newcounter{notes}%[page]   %Le 2eme argument fait reinitialiser les numeros de notes a chaque page
\newcommand{\ignore}[1]{}

\newtheorem{theorem}{Theorem}
\newtheorem{proposition}[theorem]{Proposition}
\newtheorem{corollary}[theorem]{Corollary}
\newtheorem{lemma}[theorem]{Lemma}

\newtheorem{convention}[theorem]{Convention}

\newtheorem{observation}[theorem]{Observation}

\theoremstyle{definition}
\newtheorem{definition}[theorem]{Definition}
\newtheorem{remark}[theorem]{Remark}

\newtheorem{notation}[theorem]{Notation}

\newtheorem*{acknowledge}{Acknowledgements}

\newtheoremstyle{theoremwithref}{}{}{\itshape}{}{\bfseries}{.}{.5em}{#1 #2 #3}
\theoremstyle{theoremwithref}

\newcommand{\R}{\mathbb{R}}
\newcommand{\Z}{\mathbb{Z}}
\newcommand{\N}{\mathbb{N}}

\newcommand{\PSL}{\mathrm{PSL}(2,\mathbb{R})}

\newcommand{\SO}{\mathrm{SO}}

\newcommand{\Hom}{\mathrm{Hom}}

\newcommand{\HOS}{\mathrm{Homeo}^+(S^1)}
\newcommand{\HOZ}{\mathrm{Homeo}_\Z^+(\R)}
\newcommand{\rot}{\mathrm{rot}}
\newcommand{\rotild}{\widetilde{\mathrm{rot}}}
\newcommand{\eu}{\mathrm{eu}}
\newcommand{\id}{\mathrm{id}}
\newcommand{\Fix}{\mathrm{Fix}}

\DeclareMathOperator{\Homeo}{Homeo}

\title[Fuchsian actions from rigidity]{A characterization of Fuchsian actions by topological rigidity}

\author{Kathryn Mann}
\address{Department of Mathematics, Brown University, 151 Thayer Street, Providence, RI 02912, USA
}
\email{mann@math.brown.edu}

\author{Maxime Wolff}
\address{Sorbonne Universit\'es, UPMC Univ.\ Paris 06, Institut de Math\'ematiques
de Jussieu-Paris Rive Gauche, UMR 7586, CNRS, Univ. Paris Diderot, Sorbonne
Paris Cit\'e, 75005 Paris, France}
\email{maxime.wolff@imj-prg.fr}

\begin{document}

\maketitle
\numberwithin{theorem}{section}
%\numberwithin{equation}{section}

\begin{abstract}
We prove that any rigid representation of $\pi_1\Sigma_g$ in $\HOS$ with Euler
number at least $g$ is necessarily semi-conjugate to a discrete,
faithful representation into $\PSL$.
Combined with earlier work of Matsumoto, this precisely characterizes
Fuchsian actions by a topological rigidity property.
Though independent, this work can be read as an introduction to the companion
paper \cite{RigidGeom} by the same authors.\\
\vspace{0.2cm}
%MSC Classification: %58D29, 57M05, 20H10, 30F60. 

\end{abstract}

%%%%%%%%%%%%%%%%%%%%%%%%%%%%%%%%%%%%%%%%%%%%%%%%%%%
\section{Introduction} \label{sec:Intro}

Let $\Sigma_g$ be a surface of genus $g \geq 2$, and let
$\Gamma_g = \pi_1(\Sigma_g)$.
The \emph{representation space} $\Hom(\Gamma_g, \HOS)$ is the set of all
actions of $\Gamma_g$ on $S^1$ by orientation-preserving homeomorphisms,
equipped with the compact-open topology.
This is also the space of \emph{flat topological circle bundles} over
$\Sigma_g$, or equivalently, the space of circle bundles with a foliation
transverse to the fibers.  
The \emph{Euler class} of a representation $\rho \in \Hom(\Gamma_g, \HOS)$ is
defined to be the Euler class of the associated bundle, and the
\emph{Euler number} $\eu(\rho)$ is the integer obtained by pairing the Euler
class with the fundamental class of the surface.
The classical Milnor-Wood inequality \cite{Milnor, Wood} is the statement
that the absolute value of the Euler number of a flat bundle is bounded by
the absolute value of the Euler characteristic of the surface.

While the Euler number determines the topological type of a flat $S^1$ bundle, 
it  doesn't nearly determine its flat structure -- except in the special case where the Euler number is maximal, i.e. equal to $\pm(2g-2)$.    
In this case, a celebrated result of Matsuomoto states that for any 
representation $\rho$ with $\eu(\rho) = \pm(2g-2)$, there
there is a continuous, degree one,
monotone map $h: S^1 \to S^1$ such that
\begin{equation} \label{eq:Semiconj}
h \circ \rho = \rho_F \circ h
\end{equation}
where $\rho_F$ is a discrete, faithful representation of $\Gamma_g$ into
$\PSL$ (i.e. a bijection onto a cocompact lattice).
We view $\PSL \subset \HOS$ via the action on $\R \mathrm{P}^1 \cong S^1$
by M\"obius transformations.

An important consequence of Matsumoto's theorem is that representations with
maximal Euler number are dynamically {\em stable} or {\em rigid} in the
following sense.

\begin{definition} \label{rig def}
  Let $\Gamma$ be a discrete group. A representation $\rho: \Gamma \to \HOS$
  is called \emph{path-rigid} if its path-component in $\Hom(\Gamma, \HOS)$
  consists of a single semi-conjugacy class.
\end{definition}

\emph{Semi-conjugacy} is the equivalence relation generated by the property
shared by $\rho$ and $\rho_F$ in \eqref{eq:Semiconj} above; we recall the
precise definition in Section \ref{prelimsec}.
As semi-conjugacy classes are connected in $\Hom(\Gamma_g, \HOS)$,
path-rigid representations are precisely those whose path-component is as
small as possible.
In fact, Matsumoto's work implies that maximal representations have the
the stronger property (called ``rigid''
in \cite{RigidGeom}) of defining an isolated point in
the \emph{character space} for representations of $\Gamma_g$ into $\HOS$;
a notion of rigidity that generalizes to representations into arbitrary
topological groups. However, for simplicity, we will not
define character spaces here and refer the reader to \cite{RigidGeom} for
details.

This paper proves the following converse to Matsumoto's rigidity result.

\begin{theorem}\label{thm:AuMoinsG}
  Let $\rho: \Gamma_g \to \HOS$ be a path-rigid representation, with
  $|\eu(\rho)| \geq g$. Then $\eu(\rho)$ is maximal, i.e.
  $| \eu(\rho)| = 2g-2$, and $\rho$ is semi-conjugate to a discrete,
  faithful representation into $\PSL$.
\end{theorem}

Thus, Fuchsian representations are characterized among all representations
with Euler number at least $g$ by path-rigidity.
The assumption $|\eu(\rho)| \geq g$ is not superfluous -- it is shown
in \cite{KatieInvent} that many representations with $|\eu(\rho)| \leq g-1$
are path-rigid as well. However, our assumption can be replaced with an
\emph{a priori} strictly weaker assumption on the dynamics of $\rho$,
phrased in terms of rotation numbers of elements, as follows.

\begin{theorem}\label{thm:Main}
  Suppose $\rho\colon\Gamma_g \to \HOS$ is path-rigid. If there exist based
  simple closed curves $a, b \in \Gamma_g$ with intersection number 1 and
  such that $\rotild[\rho(a), \rho(b)] = \pm 1$, then $\eu(\rho) = \pm(2g-2)$,
  and $\rho$ is semi-conjugate to a Fuchsian representation.
\end{theorem}

The hypothesis $\rotild[\rho(a),\rho(b)]=\pm 1$ is equivalent to the statement
that the restriction of the representation to the torus defined by $a$ and $b$
is semi-conjugate to a geometric one (see \cite{MatsBP}). Thus, one can think
of the statement above as a local--to--global result: the local condition that
a torus is Fuchsian, together with path-rigidity, implies the global statement
that the representation is Fuchsian.

\subsection{Geometric representations}   
If $M$ is a manifold, and $\Gamma$ a discrete group, a representation
$\rho\colon\Gamma \to \Homeo^+(M)$ is called \emph{geometric} if it is a
bijection onto a cocompact lattice in a transitive, connected Lie group in
$\Homeo(M)$.

It is not difficult to prove that the transitive Lie subgroups of $\HOS$ are
precisely $\SO(2)$ and the central extensions of $\PSL$ by finite cyclic
groups (see~\cite{Ghys01}). Thus, the geometric representations
$\Gamma_g \to \HOS$ are either Fuchsian or obtained by lifting a Fuchsian
representation to one of these central extensions of $\PSL$.
The main result of \cite{KatieInvent} implies that all geometric
representations are path-rigid (in fact, the proof shows the stronger result
that they are rigid in the character space sense), and it was conjectured
there the converse held as well.

This paper proves that conjecture under an additional assumption, which rules
out the case that $\rho$ could be a lattice in a nontrivial central extension
of $\PSL$. This assumption greatly simplifies the situation,
allowing us to give a short proof of the converse to Matsumoto's result.
The general converse is the subject of our recent paper \cite{RigidGeom}.
Though self-contained and independent, the present work is also intended to
serve as an introduction to the ideas in \cite{RigidGeom}, communicating some
of the philosophy of the proof in a simplified setting that avoids much of the
technical nightmare. In this spirit, we have taken care to make the proof here
as explicit and elementary as possible.

\subsection{Outline}
In Section \ref{prelimsec} we recall standard material on dynamics of groups
acting on the circle, including rotation numbers and the Euler number for
actions of surface groups. We then introduce important tools for the proof
of Theorem \ref{thm:Main}, and give a quick proof that Theorem \ref{thm:Main}
implies Theorem \ref{thm:AuMoinsG}.

Sections~\ref{hypsec} through~\ref{subsurface sec} are devoted to the proof
of Theorem~\ref{thm:Main}. The general strategy is as follows. Given a
representation $\rho$ satisfying the hypotheses of Theorem~\ref{thm:Main},
we show that:
\begin{enumerate}[1.]
\item After modifying $\rho$ by a semi-conjugacy, there exists
  $a \in \Gamma_g$ represented by a nonseparating simple closed curve and
  such that $\rho(a)$ is \emph{hyperbolic}, meaning that it is conjugate to
  a hyperbolic element of $\PSL$.
\item Using step 1, we then show that (again after semi-conjugacy of $\rho$),
  \emph{any} $\gamma \in \Gamma_g$ represented by a nonseparating simple
  closed curve has the property that $\rho(\gamma)$ is hyperbolic.
  These two first steps are done in Section~\ref{hypsec}.
\item Next, we start to ``reconstruct the surface'', showing that the
  arrangement of attracting and repelling points of hyperbolic elements
  $\rho(\gamma)$, as $\gamma$ ranges over simple closed curves, mimics that
  of a Fuchsian represenation into $\PSL$.
  This is carried out in Section \ref{fixsec}.
\item Finally, in Section \ref{subsurface sec} we show that the restriction
  of $\rho$ to small subsurfaces is semi-conjugate to a Fuchsian
  representation; this is then improved to a global result by additivity of
  the \emph{relative Euler class}.
\end{enumerate}

Throughout this paper, whenever we say ``deformation'', we mean deformation
along a continuous path in $\Hom(\Gamma_g, \HOS)$.

\begin{acknowledge}
This work was started at MSRI during spring 2015 at a program supported by
NSF grant 0932078.
Both authors also acknowledge support of the U.S. National Science Foundation
grants DMS 1107452, 1107263, 1107367
``RNMS: Geometric Structures and Representation Varieties'' (the GEAR network).
K. Mann was partially supported by NSF grant DMS-1606254, and thanks the
Inst. Math. Jussieu and Foundation Sciences Math\'ematiques de Paris.
This work was finished while M. Wolff was visiting the Universidad de la
Rep\'ublica, Montevideo, Uruguay.
\end{acknowledge}

%-----------------------------------------------------------------------
\section{Preliminaries}  \label{prelimsec}

This section gives a very quick review of basic concepts used later in the
text. The only material that is not standard is the
\emph{based intersection number} discussed in Section \ref{ssec:Tori}.

%------------------
\subsection{Rotation numbers and the Euler number}\label{ss:rot}

Let $\HOZ$ denote the group of homeomorphisms of $\R$ that commute with
integer translations, this is a central extension of $\HOS$ by $\Z$.
The primary dynamical invariant of such homeomorphisms is the translation
or rotation number, whose use can be traced back to work of
Poincar\'e \cite{Poincare}.  
\begin{definition}[Poincar\'e]
  Let $\widetilde{g} \in \HOZ$ and $x \in \R$. The \emph{translation number}
  $\rotild(\widetilde{g})$ is defined by
  $\rotild(\widetilde{g}) := \lim \limits_{n \to \infty}
  \frac{\widetilde{g}^n(x)}{n}$,
  where $x$ is any point in $\R$.
  For $g \in \HOS$, the ($\R/\Z$-valued) \emph{rotation number} of $g$ is
  defined by $\rot(g) := \rotild(\widetilde{g}) \, \mathrm{mod}\, \Z$, where
  $\widetilde{g}$ is any lift of $g$.
\end{definition} 

It is a standard exercise to show that these limits exist, and are
independent of the choice of point $x$. Note that $\rot(g)$ is also
independent of the choice of
lift $\widetilde{g}$, and that $\rot$ is invariant under conjugacy.
(In fact it is invariant under semi-conjugacy as well.)

One way of defining the Euler number of a representation is in terms of
translation numbers. This was perhaps first observed by Milnor and
Wood \cite{Milnor, Wood}, who showed the following. For the purposes of this
work, the reader may take this as the \emph{definition} of Euler number.

\begin{proposition}\label{prop:DefMilnor}
  Consider a standard presentation 
  \[\Gamma_g = \langle a_1, b_1, \ldots, a_g, b_g \mid
  \prod_i [a_i, b_i] \rangle. \]
  Let $\rho \in \Hom(\Gamma_g, \HOS)$, and let $\widetilde{\rho(a_i)}$ and
  $\widetilde{\rho(b_i)}$ be any lifts of $\rho(a_i)$ and $\rho(b_i)$ to
  $\HOZ$. Then the \emph{Euler number} $\eu(\rho)$ is given by
  $$\eu(\rho)= \rotild \left( [\widetilde{\rho(a_1)}, \widetilde{\rho(b_1)}]
  \cdots [\widetilde{\rho(a_g)}, \widetilde{\rho(b_g)}] \right).$$
\end{proposition} 

\noindent Note that, for any $f$ and $g \in \HOS$, the value of the
commutator $[\widetilde{f}, \widetilde{g}]$ is independent of the choice
of lifts $\widetilde{f}$ and $\widetilde{g}$ in $\HOZ$.

As remarked in the introduction, the Milnor--Wood inequality is the statement
that $| \eu(\rho) | \leq 2g-2$.
For a simpler surface as a one-holed torus, we have the following variation,
which was essentially proved in \cite{Wood};
see the discussion following Lemma~\ref{rk:EulerBound}
below.

\begin{lemma}\label{max comm}
  Let $[\widetilde{f}, \widetilde{g}] \in \HOZ$ be a commutator.
  Then $-1 \leq \rotild([\widetilde{f}, \widetilde{g}]) \leq 1$.
\end{lemma}

Though unimportant in the preceding remarks, in what follows we will need
to fix a convention for commutators and group multiplication.
\begin{convention}
  We read words in $\Gamma_g$ from right to left, so that group multiplication
  coincides with function composition. (This is convenient for dealing with
  representations to $\HOS$.) We set the notation for a commutator as
  \[[a,b]:= b^{-1} a^{-1} b a. \]
\end{convention}

%----------------------
\subsection{Dynamics of groups acting on $S^1$}

The material in this section is covered in more detail in \cite{Ghys01}
and \cite{KatieSurvey}.

\begin{definition}[Ghys \cite{Ghys87}]\label{def:SC}
  Let $\Gamma$ be a group. Two representations $\rho_1$, $\rho_2$ in
  $\Hom(\Gamma, \HOS)$ are \emph{semi-conjugate} if there is a monotone
  (possibly non-continuous or non-injective) map
  $\widetilde{h}\colon\R\to\R$ such that
  $\widetilde{h}(x+1)=\widetilde{h}(x)+1$ for all $x\in\R$, and such that,
  for all $\gamma \in \Gamma$, there are lifts $\widetilde{\rho_1(\gamma)}$
  and $\widetilde{\rho_2(\gamma)}$ such that
  $\widetilde{h}\circ\widetilde{\rho_1(\gamma)}=
  \widetilde{\rho_2(\gamma)}\circ \widetilde{h}$.
\end{definition}

Ghys \cite{Ghys87} showed that semi-conjugacy is an equivalence relation
on $\Hom(\Gamma, \HOS)$ (see also \cite{KatieSurvey} for an exposition of the
proof). In fact, it is the equivalence relation generated by the relationship
shared by $\rho$ and $\rho_F$ in Equation \eqref{eq:Semiconj} of Section \ref{sec:Intro}.

The next proposition states a useful dynamical trichotomy for groups acting
on the circle, which in particular can be used to explain when a
semi-conjugacy map can be taken to be continuous.  As it is classical, we do
not repeat the proof; the reader may refer to \cite[Prop.~5.6]{Ghys01}.

\begin{proposition}\label{prop:Trichotomy}
  Let $G \subset \HOS$. Then exactly one of the following holds:
  \begin{enumerate}[i)]
  \item $G$ has a finite orbit.
  \item $G$ is \emph{minimal}, meaning that all orbits are dense.
  \item There is a unique compact $G$-invariant subset of $S^1$ contained in
    the closure of any orbit, on which $G$ acts minimally. This set is
    homeomorphic to a Cantor set and called the
    \emph{exceptional minimal set} for $G$.
  \end{enumerate}
\end{proposition}

In case $iii)$, defining $h$ to be a map that collapses each interval in the
complement of the exceptional minimal set to a point gives the following
(we leave the proof as an exercise).

\begin{proposition} \label{prop:MinimalConj}
  Let $\rho\colon G \to \HOS$ be a homomorphism such that $\rho(G)$ has an
  exceptional minimal set. Then $\rho$ is semi-conjugate to a homomorphism
  $\nu$ whose image is minimal. Moreover, provided that $\nu$ is minimal,
  any semi-conjugacy $h$ to any representation $\rho'$ such that
  $h \circ \rho' = \nu \circ h$ is necessarily continuous.
\end{proposition}

We will make frequent use of the following two consequences of
Proposition \ref{prop:MinimalConj}.

\begin{corollary}\label{cor:MinConj}
  Suppose that $\rho$ and $\rho'$ are semi-conjugate representations.
  If both $\rho$ and $\rho'$ are minimal, then they are \emph{conjugate}.
\end{corollary}

\begin{corollary}\label{cor:MinRep}
  Let $\rho \in \Hom(\Gamma_g, \HOS)$ be a path-rigid representation.
  Then $\rho$ is semi-conjugate to a minimal representation.
\end{corollary}

\begin{proof}
  Corollary \ref{cor:MinConj} follows immediately from
  Proposition~\ref{prop:MinimalConj}. We now prove Corollary~\ref{cor:MinRep}.
  Using Propositions~\ref{prop:Trichotomy} and~\ref{prop:MinimalConj}, it
  suffices to show that a representation with a finite orbit is not path-rigid.
  If $\rho$ has a finite orbit, then we may perform the Alexander trick to
  continuously deform $\rho$ into a representation with image in $\SO(2)$.
  As $\Hom(\Gamma_g, \SO(2)) = \SO(2)^{2g}$, the representation $\rho$ can be
  deformed arbitrarily within this space, in particular to a non
  semi-conjugate representation.
\end{proof}
Following Corollary \ref{cor:MinRep}, in the proof of Theorem~\ref{thm:Main}
we will occasionally make the (justified) assumption that a path-rigid
representation $\rho$ is also minimal.

%------------------
\subsection{Deforming actions of surface groups}

Let $\gamma\in\Gamma_g$ be a based, simple loop. Cutting $\Sigma_g$ along
$\gamma$ decomposes $\Gamma_g$ into an amalgamated product
$\Gamma_g=A\ast_{\langle\gamma\rangle}B$ if $\gamma$ is separating, and an
HNN-extension $A\ast_{\langle\gamma\rangle}$ if not.
In both cases, $A$ and $B$ are free groups.
As there is no obstruction to deforming a representations of a free group
into any topological group, deforming a representation
$\rho\colon\Gamma_g\to\HOS$ amounts to deforming the restriction(s) of $\rho$
on $A$ (and $B$, if $\gamma$ separates), subject to the single constraint
that these should agree on $\gamma$.

The following explicit deformations are analogous to special cases of
\emph{bending deformations} from the theory of quasi-Fuchsian and Kleinian
groups.

\begin{definition}(Bending deformations)\label{def:Bending}
  \begin{enumerate}
  \item \textit{Separating curves.} Let $\gamma=c \in \Gamma_g$ represent a
    separating simple closed curve with
    $\Gamma_g = A \ast_{\langle c \rangle} B$.
    Let $c_t$ be a one-parameter group of homeomorphisms commuting with
    $\rho(c)$. Define $\rho_t$ to agree with $\rho$ on $A$, and to be equal
    to $c_t \rho c_t^{-1}$ on $B$.
  \item \textit{Nonseparating curves.} Let $\gamma=a$, and let $b\in\Gamma_g$
    with $i(a,b)= -1$.   Let $c=[a,b]$, writing again
    $\Gamma_g = A \ast_{\langle c \rangle} B$.
    Let $a_t$ be 1-parameter group commuting with $\rho(a)$ and define
    $\rho_t$ to agree with $\rho$ on $B$, and on $\langle a \rangle$,
    and define $\rho_t(b) = a_t\rho(b)$.
  \end{enumerate}
  In both cases, we call this deformation a {\em bending along $\gamma$}.
\end{definition}

In particular, if $\gamma_t$ is a one-parameter group with $\gamma_1=\gamma$,
then the deformation given above is the precomposition of $\rho$ with
${\tau_\gamma}_\ast$, where $\tau_\gamma$ is the Dehn twist along $\gamma$.
Note that we have made a specific (though arbitrary) choice realizing the
Dehn twist as an automorphism of $\Gamma_g$. This will allow us to do
specific computations, for which having a twist defined only up to inner
automorphism would not suffice. (See the discussion on based curves in the
next subsection for more along these lines.)
While not every $f \in \HOS$ embeds in a one-parameter group, every element
with at least one fixed point does, and this is the situation in which we
will typically apply bending deformations in this article.

The next corollary is used frequently in the proof of Theorem~\ref{thm:Main}.

\begin{corollary} \label{minconj}
  Suppose that $\rho$ is a path-rigid, minimal representation.
  Let $\rho_t$ be a bending deformation along $a$, using a deformation $a_t$,
  with $a_1 = a^N$ for some $N \in \Z$. Then $\rho_1$ is conjugate to $\rho$.
\end{corollary}

\begin{proof}
  By the discussion above, $\rho_1$ agrees with precomposition of $\rho$
  with an automorphism of $\Gamma_g$, so has the same image.
  Corollary \ref{cor:MinConj} now implies that these are conjugate.
\end{proof} 

%--------------------------------
\subsection{Based curves, chains, and Fuchsian tori}\label{ssec:Tori}

If $a$ and $b$ are simple closed curves on $\Sigma_g$, the familiar
\emph{geometric intersection number} is the minimum value of $|a' \cap b'|$,
where $a'$ and $b'$ are any curves freely homotopic to $a$ and $b$
respectively. It is well known that if $a$ and $b$ are nonseparating simple
closed curves with geometric intersection number $1$, then there is a
subsurface $T \subset \Sigma$ homeomorphic to a torus with one boundary
component with fundamental group (freely) generated by $a$ and $b$.
(See e.g. \cite{Primer} Section 1.2.3.)

As mentioned earlier, the fact that we are working with specific
representations, rather than conjugacy classes of elements, forces us to take
basepoint and orientation of curves into account. Although our notation
$\Gamma_g = \pi_1\Sigma_g$ does not mention a basepoint, all elements of
$\pi_1\Sigma_g$ will henceforth always be assumed based, and we will use the
following variation on the standard definition of intersection number.

\begin{definition}[Based intersection number]
  Let  $a,b\in \Gamma_g$. We write $i(a,b)=0$ if we can represent $a$ and $b$
  by differentiable maps $a,b\colon[0,1]\to\Sigma_g$, based at the base point,
  whose restrictions to $[0,1)$ are injective, and such that the cyclic order
  of their tangent vectors at the base point is either
  $(a'(0),-a'(1),b'(0),-b'(1))$ or $(a'(0),-a'(1),-b'(1),b'(0))$, or the
  reverse of one of these. \\
  If instead the cyclic order of tangent vectors is
  $(a'(0),b'(0),-a'(1),-b'(1))$ or the reverse, we write $i(a,b)=1$ and
  $i(a,b)=-1$ respectively.
\end{definition}

This is a somewhat ad-hoc definition.
In particular, $i(a,b)$ is left undefined for many pairs $(a, b)$.

For more than two curves, the following definition will be convenient.
\begin{definition}
  A {\em directed $k$-chain}, in $\Sigma_g$, is a $k$-uple
  $(\gamma_1,\ldots,\gamma_k)$ of elements of $\Gamma_g$, such that the
  oriented curves $\gamma_i$ may be homotoped simultaneously (rel. the base
  point) in order to realize an embedding (possibly orientation-reversing,
  but respecting the orientations of the edges) of the graph shown in
  Figure~\ref{fig:Chaine}.
\end{definition}
In particular, $i(\gamma_i,\gamma_j)=\pm 1$ if $|j-i|=1$, and $0$ otherwise.
\begin{figure}[hbt]
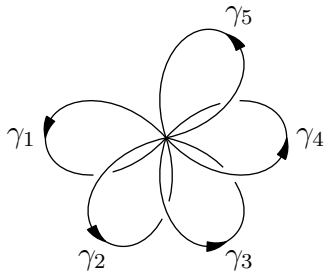

\begin{asy}
  import geometry;
  real t = 48;
  real T=60;
  //% Les courbes
  path courba = (0,0){dir(180-t)}..(-45,0)..(0,0){dir(t)};
  path courbb = rotate(T)*courba, courbc = rotate(T)*courbb;
  path courbd = rotate(T)*courbc, courbe = rotate(T)*courbd;
  //% Les intersections
  point pointab = intersectionpoint(subpath(courba,0.1,1.9),courbb);
  point pointbc = intersectionpoint(subpath(courbb,0.2,1.8),courbc);
  point pointcd = intersectionpoint(subpath(courbc,0.2,1.8),courbd);
  point pointde = intersectionpoint(subpath(courbd,0.2,1.8),courbe);
  //% On dessine
  draw("$\gamma_1$",courba,Arrow(Relative(0.5))); dot(pointab,7pt+white);
  draw("$\gamma_2$",courbb,Arrow(Relative(0.5))); dot(pointbc,7pt+white);
  draw("$\gamma_3$",courbc,Arrow(Relative(0.5))); dot(pointcd,7pt+white);
  draw("$\gamma_4$",courbd,Arrow(Relative(0.5))); dot(pointde,7pt+white);
  draw("$\gamma_5$",courbe,Arrow(Relative(0.5)));
\end{asy}
\caption{A directed chain of length $5$}
\label{fig:Chaine}
\end{figure}
Note that we do not require that the embedding be $\pi_1$-injective. For
example, whenever $i(\gamma_1,\gamma_2)=1$, then
$(\gamma_1,\gamma_2,\gamma_1^{-1})$ is a directed $3$-chain, rather degenerate.

These $k$-chains will be useful especially to study bending deformations that
realize sequences of Dehn twists.
Whenever $(\gamma_1,\ldots,\gamma_k)$ is a directed $k$-chain,
the Dehn twist along the curve $\gamma_i$ may be described by an automorphism
of $\Gamma_g$ leaving invariant the elements $\gamma_j$ for $|j-i|\geq 2$
and $j=i$, and mapping $\gamma_{i-1}$ to $\gamma_i^{-1}\gamma_{i-1}$, and
 $\gamma_{i+1}$ to $\gamma_{i+1}\gamma_i$.

\begin{notation}
  Let $i(a,b) = \pm1$. Then their commutator $[a,b]$ bounds a genus 1
  subsurface (well defined up to homotopy) containing $a$ and $b$.  We denote 
  this surface by $T(a,b)$.
\end{notation}

\begin{definition} \label{def:StandardF}
  We call any representation $\rho\colon \pi_1(T(a,b)) \to \PSL$ arising
  from a hyperbolic structure of infinite volume on $T(a,b)$ a
  \emph{standard Fuchsian representation of a once-punctured torus}.
  Similarly, we say that $\rho: \Gamma_g \to \PSL$ is
  \emph{standard Fuchsian} if it comes from a hyperbolic structure on
  $\Sigma_g$.
\end{definition}

\begin{convention}
  We assume $\Sigma_g$ is oriented, hence standard Fuchsian representations
  of $\Gamma_g$ have Euler number $-2g+2$, and are all conjugate in $\HOS$.
  Similarly, $T(a,b)$ inherits and orientation, so all its standard Fuchsian
  representations are conjugate in $\HOS$.
\end{convention}

\begin{definition}\label{def:FuchTorus}
  We say that $\rho: \Gamma_g \to \HOS$  has a {\em Fuchsian torus} if there
  exist two simple closed curves $a,b\in\Gamma_g$, with $i(a,b) = \pm1$ and
  such that $\rotild([\rho(a),\rho(b)])= \pm1$.
\end{definition}

The terminology ``Fuchsian torus'' in Definition~\ref{def:FuchTorus} comes
the following observation of Matsumoto.

\begin{observation}[\cite{Matsumoto}]\label{obs:Matsumoto}
  Let $\alpha, \beta\in\HOS$ be such that
  $\rotild([\alpha,\beta])=\pm 1$.
  Then $\alpha$ and $\beta$ generate a free group, and, up to reversing the
  orientation of $S^1$, this group is semi-conjugate to a standard Fuchsian
  representation of a one-holed torus $T(a,b)$ with $\rho(a) = \alpha$ and
  $\rho(b) = \beta$.
\end{observation}

The proof of Observation \ref{obs:Matsumoto} not difficult, an easily
readable sketch is given in~\cite[\S~3]{MatsBP}.

The next lemma shows the existence of such a torus is guaranteed, provided the
absolute value of the Euler number of a representation is sufficiently high.

\begin{lemma}\label{Fuchsian lem}
  If $|\eu(\rho)| \geq g$ then $\rho$ has a Fuchsian torus.
\end{lemma}

\begin{proof}
  If $\eu(\rho) \geq g$, then conjugating $\rho$ by an orientation-reversing
  homeomorphism of $S^1$ gives a representation with Euler number at most $-g$.
  Thus, we may assume that $\eu(\rho) \leq -g$.
  Let $f\in\HOZ$. It is an easy consequence of the definition of $\rotild$
  that $\rotild(f)>0$ if and only if $f(x) > x$ for all $ x\in\R$.
  Hence if $f_1,\ldots,f_g\in\HOZ$ satisfy $\rotild(f_i)>0$
  for all $i$, then $\rotild(f_1\cdots f_g)>0$.
  
  By composing such $f_i$ by translation by $-1$, which is central in $\HOZ$,
  we deduce that if
  $\rotild(f_i)>-1$ for all $i$ then $\rotild(f_1\cdots f_g)>-g$.
  Now let $\rho$ be a representation, and let
  $f_i = [ \widetilde{\rho(a_i)}, \widetilde{\rho(b_i)}]$.
  Then the inequality $\eu(\rho)\leq -g$ implies
  $\rotild(f_i)\leq -1$ for some $i$.
  As the maximum absolute value of the rotation number of a commutator is $1$
  by Lemma \ref{max comm}, we in fact have $\rotild(f_i) = -1$ for some $i$.
\end{proof}

Lemma \ref{Fuchsian lem} immediately shows that Theorem~\ref{thm:Main}
implies Theorem~\ref{thm:AuMoinsG}.
The rest of this work is devoted to the proof of Theorem~\ref{thm:Main}.

%---------------------------------
\section{Step 1: Existence of hyperbolic elements}\label{hypsec}

\begin{definition}\label{hyp def}
  We say a homeomorphism $f \in\HOS$ is {\em hyperbolic} if it is conjugate to
  a hyperbolic element of $\PSL$, i.e. it has one \emph{attracting fixed point}
  $f_+ \in S^1$ and one \emph{repelling fixed point} $f_- \neq f_+$ such that
  $\lim \limits_{n \to +\infty} f^n(x) = f_+$ for all $x \neq f_-$, and
  $\lim \limits_{n \to +\infty} f^{-n}(x) = f_-$ for all $x \neq f_+$.
\end{definition}

The first step of the proof of Theorem \ref{thm:Main} is to show that a rigid, minimal representation has very many hyperbolic elements.  This is the goal of this section.

\begin{lemma}\label{lemKt}
  Let $T(a,b)$ be a one-holed torus subsurface, and let $A = \pi_1T(a,b)$.
  Suppose $\rho\colon A \to \HOS$ is semi-conjugate to a standard Fuchsian
  representation, as in definition \ref{def:StandardF}.
  Then there exists a continuous deformation $\rho_t$ with $\rho_0 = \rho$
  such that
  \begin{enumerate}[i)]
  \item  $\rho_1(a)$ is hyperbolic, and 
  \item there exists a continuous family of homeomorphisms $f_t \in \HOS$
    such that $\rho_t([a,b]) = f_t \rho([a,b]) f_t^{-1}$ for all $t$.
  \end{enumerate}
\end{lemma}

\begin{proof}
Let $c$ denote the commutator $[a,b]$.
Let $\rho_0$ denote the minimal representation (unique up to conjugacy) that is semi-conjugate to $\rho$.  Since $\rho$ is semi-conjugate to a standard Fuchsian representation, $\rho_0$ is the representation corresponding to the \emph{finite volume} hyperbolic structure on $T(a,b)$. 
By Observation \ref{obs:Matsumoto} and Proposition \ref{prop:MinimalConj}, there is a continuous map $h:S^1 \to S^1$, collapsing each component of the exceptional minimal set for $\rho$ to a point, satisfying  $h \rho = \rho_0 h$.  Let $x_+$ and $x_-$ be the endpoints of the axis of $\rho_0(a)$, and $X_+$ and $X_-$ the pre-images under $h$ of their orbits $\rho(A)x_+$ and $\rho(A)x_-$.

Note that $X_+$ and $X_-$ are both $\rho(A)$-invariant sets and their images under $h$ are the attractors (respectively, repellers) of closed curves in $T(a,b)$ conjugate to $a$.  Moreover, for this reason, $X_+$ and $X_-$ lie in a single connected component of $S^1 \smallsetminus \Fix(\rho(c))$, and the interiors of the intervals that make up $X_+$ and $X_-$ are disjoint from the exceptional minimal set of $\rho$.    

Define a continuous family of continuous maps $h_t: S^1 \to S^1$, with $h_0 = \id$, as follows:  
We define $h_t$ to be the identity on the complement of the connected component of $S^1 \smallsetminus \Fix(\rho(c))$ containing $X_+$ and $X_-$, and for each interval $I$ of $X_+$ or of $X_-$, have $h_t$ be a homotopy contracting that interval so that $h_1(I)$ is a point.    

To make this precise, one needs to fix an identification of the target of $h_t$ with the standard unit circle.  Let $J$ be the connected component of $S^1 \smallsetminus \Fix(\rho(c))$ that contains the exceptional minimal set of $\rho(A)$.  Define $h_t$ to rescale the length of each connected component of $X_+$ or $X_-$ by a factor of $(1-t)$ and rescale the complement of $X_+ \cup X_-$ in $J$ so that the total length of $J$ remains unchanged.  This gives us the desired map $h_t$ which is the identity outside of $J$, and contracts intervals of $X_+$ and $X_-$ to points.    

Now define $\rho_t$ by $h_t \rho(g) h_t^{-1}= \rho_t(g)$ for $t \in [0,1)$.  
We claim that there is a unique $\rho_1(g)$ satisfying  $h_1 \rho(g) = \rho_1(g) h_1$.  Indeed, $\rho(g)$ permutes the complementary intervals of the exceptional minimal set for $\rho$, so letting $h_1^{-1}(x)$ denote the pre-image of $x$ by $h_1$ (which is either a point or an open interval complementary to the exceptional minimal set), $h_1 \rho(g) h_1^{-1}(x)$ is always a single point, and $h_1 \rho(g) h_1^{-1}$ defines in this way a homeomorphism, which we denote by $\rho_1(g)$.  It is easily verified that $\rho_t(g)$ approaches $\rho_1(g)$ as $t \to 1$.   By construction, $\rho_1(a)$ is hyperbolic, and $\rho_t(c)$ is conjugate to a translation on the interval $J$ defined above (and hence its restriction to $J$ is conjugate to $\rho(c)|_J$), and $\rho_t(c)$ restricted to $S^1 \smallsetminus J$ agrees with $\rho_0(c)$.  
Let $f_t: S^1 \to S^1$ be a continuous family of homeomorphisms supported on $J$ that conjugate the action of $\rho_t([a,b])$ to the action of $\rho(c)$ there.  (For the benefit of the reader, justification of this step via a simple construction of such a family is given in Lemma \ref{lem:ConjFamily} below.)  Then $\rho_t(c) = f_t \rho(c) f_t^{-1}$, as claimed.
\end{proof}  

\begin{lemma}\label{lem:ConjFamily}
  Let $g_t$ be a continuous family (though not necessarily a subgroup) of
  homeomorphisms of an open interval $I$, with $\Fix(g_t) \cap I = \emptyset$
  for all $t \in [0,1]$. Then there exists a continuous family of
  homeomorphisms $f_t$ such that $f_t g_t f_t^{-1} = g_0$ for all~$t$.
\end{lemma}

\begin{proof}
  Fix $x$ in the interior of $I$, and let $D_t := [x, g_t(x)]$ be a
  fundamental domain for the action of $g_t$. Define the restriction of $f_t$
  to $D_0$ be the (unique) affine homeomorphism $D_0 \to D_t$, and extend
  $f_t$ equivariantly to give a homeomorphism of~$I$.
\end{proof}

\begin{corollary}\label{Kt}
  Let $\rho: \Gamma_g \to \HOS$. Suppose that $a$ and $b$ are simple closed
  curves in $\Gamma_g$ with $i(a,b)=\pm 1$ and
  $\rotild([\widetilde{\rho(a)},\widetilde{\rho(b)}])=\pm 1$.
  Then there exists a deformation $\rho'$ of $\rho$ such that $\rho'(a)$ is
  hyperbolic. If additionally $\rho$ is assumed path-rigid and minimal,
  then $\rho(a)$ is hyperbolic.
\end{corollary}

\begin{proof}
Let $A$ denote the subgroup generated by $a$ and $b$ and let $c = [a,b]$, so $\Gamma_g = A \ast_{\langle c \rangle} B$.  
Let $\bar{\rho}$ denote the restriction of $\rho$ to $A$.  By Lemma \ref{lemKt}, there exists a family of representations $\bar{\rho}_t: A \to \HOS$ such that 
$\bar{\rho}_t(c) = f_t \bar{\rho}(c) f_t^{-1}$ for some continuous family $f_t \in \HOS$, and such that  $\bar{\rho}_1(a)$ is hyperbolic.  As in the bending construction, define a deformation of $\rho$ by 
$$ \rho_t(\gamma) = 
\left\{ \begin{array}{ll} 
\bar{\rho}_t(\gamma) &\text{ for } \gamma \in A \\
f_t \rho(\gamma) f_t^{-1} &\text{ for } \gamma \in B.
\end{array} \right.
$$
By construction, $\rho_t$ is a well defined representation, and $\rho_1(a) = \bar{\rho}_1(a)$ is hyperbolic.

If $\rho$ is assumed path-rigid, then this deformation $\rho'$ is semi-conjugate to $\rho$.  If $\rho$ is additionally known to be minimal, then there is a continuous map $h$ satisfying $h \circ \rho' = \rho \circ h$.  In particular, this implies that $\Fix(\rho(a)) = h \Fix(\rho'(a))$, so $\rho(a)$ has at most two fixed points.  
In this case, if $\rho(a)$ does not have hyperbolic dynamics then it has a lift to $\HOZ$ satisfying $|x - \widetilde{\rho(a)}(x)| \leq 1$ for all $x$.   But this easily implies that $| \rotild([\widetilde{\rho(a)},\widetilde{\rho(b)}]) | <  1$.  (The reader may verify this as an exercise, or see the proof of Theorem~2.2 in \cite{Matsumoto} where this computation is carried out.)   We conclude that $\rho(a)$ must be hyperbolic when $\rho$ is path-rigid and minimal.  
\end{proof}  

Having found one hyperbolic element, our next goal is to produce many others.
An important tool here, and in what follows, is the following basic
observation on dynamics of circle homeomorphisms.

\begin{observation}\label{obs:Punch}
  Let $f \in \HOS$ be hyperbolic, with attracting point $f_+$ and repelling
  point $f_{-}$, and let $g \in \HOS$. For any neighborhoods $U_-$ and $U_+$ of
  $f_-$ and $f_+$ respectively, and any neighborhoods $V_-$ and $V_+$ of $g^{-1}(f_-)$ and
  $g(f_+)$ respectively, there exists $N \in \N$ such that
  $$f^N g (S^1 \smallsetminus V_-) \subset U_+ \text{ and }
  g f^N (S^1 \smallsetminus U_-) \subset V_+.$$
\end{observation}

The proof is a direct consequence of Definition \ref{hyp def}. Note that, if
$f$ is hyperbolic, then $f^{-1}$ is as well (with attracting point $f_-$ and
repelling point $f_+$), so an analogous statement holds with $f^{-1}$ in place
of $f$ and the roles of $f_+$ and $f_-$ reversed.

We now state two useful consequences of this observation.
The proofs are elementary and left to the reader.

\begin{corollary}\label{cor:FixedPoint}
  Let $f \in \HOS$ be hyperbolic, and suppose $g$ does not exchange the fixed
  points of $f$. Then for $N$ sufficiently large, either $f^N g$ or $f^{-N}g$
  has a fixed point.
\end{corollary} 

\begin{corollary}\label{cor:FixedPoint2}
  Let $f \in \HOS$ be hyperbolic, and suppose that $g^{-1}(f_-) \neq f_+$.
  Suppose also that $f^N g$ is known to be hyperbolic for large $N$.
  Then as $N \to \infty$, the attracting point of $f^N g$ approaches $f_+$
  and the repelling point approaches $g^{-1}(f_{-})$.
\end{corollary}

With these tools in hand, we can use one hyperbolic element to find others.

\begin{proposition}\label{Presque}
  Let $\rho$ be path-rigid and minimal, and suppose that $i(a,b)=\pm1$ and
  that $\rho(a)$ is hyperbolic. Then $\rho(b)$ is hyperbolic.
\end{proposition}

\begin{proof}   
  We prove this under the assumption that $\rho(b)$ does not exchange the
  fixed points of $\rho(a)$. This assumption is justified by the next lemma
  (Lemma~\ref{exchangelem}).
  Assuming $\rho(b)$ does not exchange the points of $\Fix(\rho(a))$, by
  Corollary~\ref{cor:FixedPoint}, there exists some $N \in \Z$ such that
  $\rho(b a^N)$ has a fixed point. Since $a$ is hyperbolic, $a^N$ belongs to
  a 1-parameter family of homeomorphisms, and a bending deformation using this
  family gives a deformation $\rho_1$ of $\rho$ with
  $\rho_1(b) = \rho(b a^N)$. By Corollary~\ref{minconj}, using the fact that
  $\rho$ is minimal, $\rho_1$ and $\rho$ are conjugate.
  Thus, $\rho(b)$ has a fixed point and belongs to a 1-parameter group $b_t$.
  
  Now we can build a bending deformation $\rho'_t$ such that
  $\rho'_1(b) = \rho(b)$ and $\rho'_1(a) = \rho(ba)$. Thus,
  $\rho'_1(a^{-1}b) = \rho(a^{-1})$, which is hyperbolic.
  Since $\rho'_1$ and $\rho$ are conjugate, this means that $\rho(a^{-1}b)$
  is hyperbolic. Similarly, using the fact that $a$ belongs to a one-parameter
  group, there exists a bending deformation $\rho''_t$ with
  $\rho''_1(a^{-1}b) = \rho(b)$, and such that $\rho''_1$ is conjugate
  to $\rho$. This implies that $\rho(b)$ is hyperbolic.
\end{proof}

\begin{lemma}\label{exchangelem}
  Let $a, b \in \Gamma_g$ satisfy $i(a,b) =\pm 1$, and let
  $\rho\colon \Gamma_g \to \HOS$. Suppose that $\rho(a)$ is hyperbolic,
  and $\rho(b)$ exchanges the fixed points of $\rho(a)$.
  Then there is a deformation $\rho'$ of $\rho$ which is not semi-conjugate
  to~$\rho$.
\end{lemma}

\begin{proof}
Note first that the property that $\rho(b)$ exchanges the fixed points of $\rho(a)$ implies that $\rho(b^{-1} a^{-1} b)$ is hyperbolic with the same attracting and repelling points as $a$.  Hence $[\rho(a), \rho(b)]$ is hyperbolic with the same attracting and repelling points as well.
We now produce a deformation $\rho_1$ of $\rho$ such that $\rho_1(a)$ and $\rho_1(b)$
 are in $\PSL$, after this we will easily be able to make an explicit further deformation to a non semi-conjugate representation.   
 
First, conjugate $\rho$ so that $\rho(a) \in \PSL$ and so that the attracting and repelling fixed points of $\rho(a)$ are at $0$ and $1/2$ respectively (thinking of $S^1$ as $\R/\Z$).  Now choose a continuous path $b_t$ from $b_0 = b$ to the order two rotation $b_1: x \mapsto x+1/2$, and such that $b_t(0) = 1/2$ and $b_t(1/2) = 0$ for all $t$.   By the observation above, $[\rho(a), b_t]$ is hyperbolic with attracting fixed point $0$ and repelling fixed point $1/2$ for all $t$, so is conjugate to $\rho(a)$.  By Lemma \ref{lem:ConjFamily}, applied separately to $(0, 1/2)$ and $(1/2, 1)$, there exists a continuous choice of conjugacies $f_t$ such that $f_t [\rho(a), \rho(b)] f_t^{-1} = [\rho(a), b_t]$.   
Now to define $\rho_t$, we consider $\Gamma_g = A \ast_c B$ where $A = \langle a, b \rangle$ and $c = [a,b]$, and set
$$
\begin{array}{lll} 
\rho_t(\gamma) &=  f_t \rho(\gamma) f_t^{-1} &\text{ for } \gamma \in B \\
\rho_t(a) &= \rho(a) \\
\rho_t(b) &= b_t.
\end{array}
$$
This gives a continuous family of well defined representations, with $\rho_1(b)$ the standard order 2 rotation, and $\rho_1(a) \in \PSL$.  

To finish the proof of the lemma, it suffices to note that, for a sufficiently small deformation $b'_t$ of $\rho_1(b)$ in $\SO(2)$, the commutator $[\rho_1(a), b'_t]$ will remain a hyperbolic element of $\PSL$, as the set of hyperbolic elements is open.  Thus, there is a continuous path of conjugacies in $\HOS$ to $[\rho_1(a), b]$.  This allows us to build a deformation $\rho'$ of $\rho$ with $\rho'(b) = b'_t \in \SO(2)$, using the strategy from Corollary \ref{Kt}. Since $\rot(b'_t) \neq \rot(b) = 1/2$, it follows that $\rho'$ and $\rho$ are not semi-conjugate.
\end{proof}

The following corollary summarizes the results of this section.  

\begin{corollary}\label{ToutHyp}
  Let $\sim_i$ denote the equivalence relation on nonseparating simple closed
  curves in $\Sigma_g$ generated by $a \sim_i b$ if $i(a, b) = \pm 1$.
  Suppose $\rho\colon\Gamma_g \to \HOS$ is path-rigid, and suppose that there
  are simple closed curves $a, b$ with $i(a,b) = \pm1$ such that
  $\rotild[\rho(a), \rho(b)] = \pm 1$. Then $\rho$ is semi-conjugate to a
  (minimal) representation with $\rho(\gamma)$ hyperbolic for all
  $\gamma \sim_i a$.
\end{corollary}

\begin{remark}
  In fact, as is proved in \cite{RigidGeom}, the relation $\sim_i$ has only a
  single equivalence class! However, we will not need to use this fact here,
  so to keep the proof as self-contained and short as possible we will not
  refer to it further.
\end{remark}

%-----------------------------------------------------
\section{Step 3: configuration of fixed points} \label{fixsec}

The objective of this section is to organize the fixed points of the
hyperbolic elements in a directed $5$-chain; we will achieve this gradually by
considering first $2$-chains, then $3$-chains, and finally $5$-chains.

As in Definition \ref{hyp def}, for a hyperbolic element $f \in \HOS$ we let
$f_+$ denote the attracting fixed point of $f$, and $f_-$ the repelling point.
By ``$\Fix(f)$ \emph{separates} $\Fix(g)$'' we mean that $g_-$ and $g_+$ lie
in different connected components of $S^1 \smallsetminus \Fix(f)$.
In particular, $\Fix(f)$ and $\Fix(g)$ are disjoint.

\begin{lemma}\label{HypQuinconce}
  Let $\rho$ be path-rigid and minimal, and let $a,b$
  be simple closed curves with $i(a,b)=\pm 1$ and $\rho(a)$ hyperbolic.
  Then $\rho(b)$ is hyperbolic, and $\Fix(\rho(a))$ separates
  $\Fix(\rho(b))$ in $S^1$.
\end{lemma}

\begin{proof}
  That $\rho(b)$ is hyperbolic follows from Proposition~\ref{Presque} above.
  
  As a first step, let us prove that $\Fix(\rho(a))$ and $\Fix(\rho(b))$ are
  disjoint. Suppose for contradiction that they are not.
  Then, (after reversing orientations if needed) we have
  $\rho(a)_+=\rho(b)_+$.  Let $I$ be a neighborhood of $\rho(a)_+$ with
  closure disjoint from
  $\{ \rho(a)_-,  \rho(b)_-\}$. Then, for $N>0$
  large enough, we have $\overline{I}\subset\rho(a^{-N}b)(I)$.
  Let $\rho_t$ be a bending deformation with $\rho_0=\rho$,
  $\rho_t(a)=\rho(a)$ and $\rho_1(b)=\rho(a^{-N}b)$.
  By Corollary~\ref{minconj},
  $\rho_1(b)$ is hyperbolic. Since $\overline{I}\subset\rho(a^{-N}b)(I)$,
  its attracting fixed point is outside $I$,
  and hence $\rho_1(b)_+\neq\rho_1(a)_+$. But $\rho$ and $\rho_1$ are
  conjugate by Corollary~\ref{minconj}; this is a contradiction.
  
  Now that we know that $\Fix(\rho(a))\cap\Fix(\rho(b))=\emptyset$, we will
  prove that they separate each other.
  Suppose for contradiction that $\Fix(\rho(a))$ does not separate
  $\Fix(\rho(b))$. Up to conjugating $\rho$ by an orientation-reversing
  homeomorphism of $S^1$, and up to replacing $b$ with $b^{-1}$, the fixed
  points of $\rho(a)$ and $\rho(b)$ have cyclic order
  $( a_+, \,  a_-,  \, b_{+}, \, b_{-})$.
  (For simplicity, we have suppressed the notation $\rho$.)
  
  Fix $N \in \N$ large, and let $\rho'$ be a bending deformation of $\rho$ so
  that $\rho'(b) = \rho(a^{N})\rho(b)$, and $\rho'(a) = \rho(a)$.
  It follows from Corollaries~\ref{minconj} and~\ref{cor:FixedPoint2}
  that, if $N$ is large enough, the points
  $b_+'=\rho'(b)_+$ and $b_-'=\rho'(b)_-$ can be taken arbitrarily close,
  respectively, to $a_+$ and $\rho(b)^{-1}(a_-)$.
  Since the cyclic order of fixed points is preserved under deformation
  they are also in order $(a_+,a_-,b_+',b_-')$.
  This is incompatible with the positions
  of $a_+$ and $\rho(b)^{-1}(a_-)$, unless perhaps if
  $\rho(b)^{-1}(a_-)=a_+$.
  But if $\rho(b)^{-1}(a_-)=a_+$, then $\rho'(b)$ has no
  fixed point in $(\rho(b)^{-1}(a_+),a_+)$ as this interval is mapped into
  $(a_+,a_-)$ by $\rho(b')$. This again gives an incompatibility with the
  cyclic order.
\end{proof}
\begin{lemma}\label{lem:Ordre3Chaines}
  Let $\rho$ be path-rigid and minimal, and let $(a,b,c)$ be a directed
  $3$-chain. Suppose that $\rho(a)$ is hyperbolic, and suppose that
  $\rho(a)$ and $\rho(c)$ do not have a common fixed point.
  Then $\rho(b)$ and $\rho(c)$ are hyperbolic, and, up to reversing the
  orientation of $S^1$, their fixed points are in the cyclic order
  \[ (\rho(a)_-,\rho(b)_-,\rho(a)_+,\rho(c)_-,\rho(b)_+,\rho(c)_+). \]
\end{lemma}
\begin{proof}
  It follows from Proposition~\ref{Presque} that $\rho(b)$ and $\rho(c)$
  are hyperbolic, and from Lemma~\ref{HypQuinconce} that up to reversing
  orientation, the fixed points of $\rho(a)$ and $\rho(b)$ come in the
  cyclic order
  \[ (a_-,b_-,a_+,b_+). \]
  (For simplicity we drop $\rho$ from the notation for the fixed points of
  $a$, $b$ and $c$).
  As mentioned above, the effect of a bending deformation that realizes
  a power of a Dehn twist along $a$ is to
  leave $a$ and $c$ invariant and to replace $b$ with $ba^N$.  
  Corollary~\ref{minconj} says that the resulting representation is
  conjugate to $\rho$. By doing this with $N>0$ and $N<0$ large, we get
  representations for which $b_-'=\rho(ba^N)_-$ can be taken arbitrarily
  close to $a_+$, as well as to $a_-$. This, and Lemma~\ref{HypQuinconce}
  applied to the curves $(b,c)$, imply that the intervals $(a_+,b_+)$
  and $(b_+,a_-)$ each contain one fixed point of $c$. In order to prove
  the lemma, it now suffices to prove that the cyclic order of fixed points
  \[ (a_-,b_-,a_+,c_+,b_+,c_-) \]
  cannot occur. Suppose for contradiction that this configuration holds,
  and apply a power of Dehn twist along $b$, replacing $a$ with
  $b^{-N}a$ and $c$ with $cb^N$ (and leaving $b$ invariant), for $N>0$ large.
  Denote by $c_+'$, $c_-'$, $a_-'$ and $a_+'$ the resulting fixed points,
  ie, the fixed points of $\rho(cb^N)$ and $\rho(b^{-N}a)$ for $N>0$
  large.
  If $N$ is chosen large enough, then $c_+'$, $c_-'$ and $a_-'$ are arbitrarily close to
  $c(b_+)$, $b_-$ and $a^{-1}(b_+)$ respectively.  (See Corollary~\ref{cor:FixedPoint2}
  above.) These three points are in the inverse cyclic order as $c_+$,
  $c_-$ and $a_-$, hence, the representation $\rho'$ obtained from this
  Dehn twist cannot be conjugate to $\rho$.  This contradicts
  Corollary~\ref{minconj}, so eliminates the undesirable configuration.
\end{proof}

We are now ready to prove the main result of this section. 
\begin{proposition}\label{prop:Ordre5Chaines}
  Let $\rho$ be a path-rigid, minimal representation, and let
  $(a,b,c,d,e)$ be a directed $5$-chain in $\Sigma_g$. Suppose $\rho(a)$ is
  hyperbolic. Then, $\rho(b)$,\ldots,$\rho(e)$ are hyperbolic as well, and
  up to reversing the orientation of the circle, their fixed points are in
  the following (total) cyclic order:
  \[ (a_-,b_-,a_+,c_-,b_+,d_-,c_+,e_-,d_+,e_+).  \]
\end{proposition}
In particular, these fixed points are all distinct.  As before, for
simplicity we have dropped $\rho$ from the notation.
\begin{proof}
  That $\rho(b)$,\ldots,$\rho(e)$ are all hyperbolic follows from
  Proposition~\ref{Presque}. Next, using a bending deformation realizing a
  Dehn twist along $d$, we may change the action of $c$ 
  into $d^{-N}c$ without changing $a$, and without changing the conjugacy class of $\rho$.
  In particular, such a deformation moves the fixed points of $c$, so we can ensure that 
  $\Fix(\rho(a))$ and
  $\Fix(\rho(c))$ are disjoint. 
  
  Similarly, for any two elements in the chain
  $(a,b,c,d,e)$,
  there is a third one that intersects one but not the other.
  Thus, we may apply the same reasoning to show that all these five
  hyperbolic elements have pairwise disjoint fixed sets.
  It remains to order these fixed sets.  For this, we will
  apply Lemma~\ref{lem:Ordre3Chaines} repeatedly.
  
  First, fix the orientation of $S^1$ so that, applying Lemma~\ref{lem:Ordre3Chaines} to the directed $3$-chain
  $(a,b,c)$, we have the cyclic order of fixed points
  \[ (a_-,b_-,a_+,c_-,b_+,c_+). \]
  Now, Lemma~\ref{lem:Ordre3Chaines} applied to the directed $3$-chain
  $(b,c,d)$ implies that $d_-$ lies in the interval $(b_+,c_+)$ and $d_+$ in the
  interval $(c_+,b_-)$.   Applying the lemma to  the directed $3$-chain
  $(a,cb,d)$ implies that $d_+$ in fact lies in the interval $(c_+,a_-)$.

  The same argument using Lemma~\ref{lem:Ordre3Chaines} applied to the directed
  $3$-chains $(c,d,e)$ and $(a,dcb,e)$ shows that $e_-$ lies in the interval
  $(c_+,d_+)$ and $e_+$ in the interval $(d_+,a_-)$, as desired.  
\end{proof}

%-----------------------------

\section{Step 3: Maximality of the Euler number}  \label{subsurface sec}
In order to compute the Euler number of $\rho$, we will decompose $\Sigma_g$
into subsurfaces and compute the contribution to $\eu(\rho)$ from each part.   The proper framework for discussing this is the language of bounded cohomology: if $\Sigma$ is a surface with boundary $\partial \Sigma$, and $\rho: \pi_1(\Sigma) \to \HOS$, one obtains a characteristic number by pulling back the {\em bounded Euler class } in $H^2_b(\HOS; \R)$ to $H^2_b(\Sigma; \R) \cong H^2_b(\Sigma, \partial\Sigma; \R)$ and pairing it with the fundamental class $[\Sigma, \partial \Sigma]$.   The contribution to the Euler number of $\rho: \Sigma_g \to \HOS$ from a subsurface $\Sigma$ is simply this Euler number for the restriction of $\rho$ to $\Sigma$.  

However, in order to keep this work self-contained and elementary, we will avoid introducing the language of bounded cohomology, and give definitions in terms of rotation numbers alone.  The reader may refer to \cite[\S~4.3]{BIW} for details on the cohomological framework. 

\begin{definition}[Euler number for pants]
Let $\rho: \Gamma_g \to \HOS$, and let $P \subset \Sigma_g$ be a subsurface
homeomorphic to a pair of pants, bounded by curves $a, d$ and $(da)^{-1}$,
with orientation induced from the boundary.
Let $\widetilde{\rho(a)}$ and $\widetilde{\rho(d)}$ be any lifts of
$\rho(a)$ and $\rho(d)$ to $\HOZ$.
The \emph{Euler number of $\rho$ on $P$} is the real number
\[
\eu_P(\rho) = \rotild\left(\widetilde{\rho(a)}\right) + \rotild \left(\widetilde{\rho(b)}\right) - \rotild \left(\widetilde{\rho(d)}\widetilde{\rho(a)}\right).
\]
\end{definition}

An illustration in the case where $P$ contains the basepoint is given in Figure~\ref{figure:pantalon}.  

\begin{figure}[h]
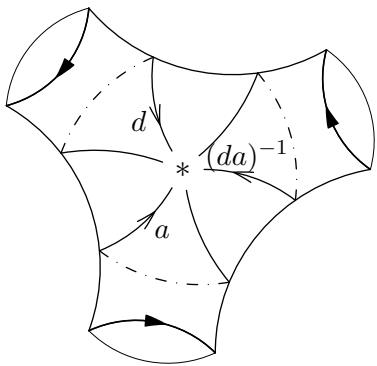

\begin{asy}
  import geometry;
  
  real long=70, ouv=40, corr1=10, corr2=35;
  
  pair p0=(long,0), p1=p0*dir(ouv), p2=p0*dir(120);
  
  path chemin1=p1{dir(ouv+180-corr1)}..{dir(120+corr1)}p2;
  path chemin2=rotate(120)*chemin1, chemin3=rotate(120)*chemin2;
  
  path bordgros1=p0{dir(90+0.5*ouv+corr2)}..p1;
  path bordfin1=p0{dir(90+0.5*ouv-corr2)}..p1;
  path bordgros2=rotate(120)*bordgros1;
  path bordfin2=rotate(120)*bordfin1;
  path bordgros3=rotate(120)*bordgros2;
  path bordfin3=rotate(120)*bordfin2;
  
  pair pa1=intersectionpoint(chemin1,(0.4*long,0)--(0.4*long,50));
  pair pa2=intersectionpoint(chemin3,(0.6*long,0)--(0.6*long,-50));
  
  path chema1=8*dir(ouv){dir(ouv)}..pa1;
  path chema2=pa1{dir(-50)}..pa2, chema3=pa2..{dir(180)}(8,0);
  
  draw(chemin1); draw(chemin2); draw(chemin3);
  
  draw(bordgros1,MidArrow); draw(bordgros1,black+0.6);
  draw(bordfin1,black+0.25); draw(bordgros2,MidArrow);
  draw(bordgros2,black+0.6); draw(bordfin2,black+0.25);
  
  draw(bordgros3,MidArrow); draw(bordgros3,black+0.6);
  draw(bordfin3,black+0.25); label("$*$",(0,0));
  
  draw(chema1); draw(chema2,dashdotted);
  draw(chema3,Arrow(SimpleHead,Relative(0.7)));
  label("\small $(da)^{-1}$\normalsize",(24,5));
  
  draw(rotate(120)*chema1);
  draw(rotate(120)*chema2,dashdotted);
  draw(rotate(120)*chema3,Arrow(SimpleHead,Relative(0.7)));
  label("\small $d$\normalsize",(24,5)*dir(120));
  
  draw(rotate(240)*chema1);
  draw(rotate(240)*chema2,dashdotted);
  draw(rotate(240)*chema3,Arrow(SimpleHead,Relative(0.7)));
  label("\small $a$\normalsize",(24,5)*dir(240));
\end{asy}
\caption{A pair of pants with standard generators of its fundamental group}
\label{figure:pantalon}
\end{figure}
Note that the number $\eu_P(\rho)$ is independent of the choice of lifts of $\rho(a)$ and $\rho(d)$.  We also allow for the possibility that the image of $P$ in $\Sigma_g$ has two boundary curves identified, so is a one-holed torus subsurface.  We may choose free generators $a, b$ for the torus, with $i(a,b) = -1$ so the torus is $T(a,b)$ and the boundary of $P$ is given by the curves $b^{-1}, a^{-1}ba$ and the commutator $[a,b]$.  Then the definition above gives $\eu_P(\rho) = \rotild [\widetilde{\rho(a)}, \widetilde{\rho(b)}]$. 

\begin{lemma}\label{rk:EulerBound}
  Let $P$ be any pants and $\rho$ a representation.
  Then $| \eu_P(\rho)| \leq 1$.
\end{lemma} 
A proof using the language of rotation numbers (consistent with our notation) can be found in \cite[Theorem~3.9]{CalegariWalker}.  
However this bound is classical and was known much earlier.   For example, the case for one-holed torus subsurfaces appears in \cite[Prop.~4.8]{Wood}, and the general case is implicit in \cite{EHN}.

More generally, if $S \subset \Sigma_g$ is any subsurface, we define the Euler number $\eu_S(\rho)$ to be the sum of relative Euler numbers over all pants in a pants decomposition of $S$.   From the perspective of bounded cohomology, it is immediate that this sum does not depend on the pants decomposition; however, since we are intentionally avoiding cohomological language, we give a short stand-alone proof. 

\begin{lemma} \label{lem:IndepPants}
For any subsurface $S \subseteq \Sigma_g$, the number $\eu_S(\rho)$ is well-defined, i.e. independent of the decomposition of $S$ into pants.  
\end{lemma}

\begin{proof} 
Any two pants decompositions can be joined by a sequence of elementary moves; namely those of type (I) and (IV) as shown in \cite{HatcherThurston}.  A type (IV) move takes place within a pants-decomposed one-holed torus $P$ so does not change the value of $\eu_P$, which is simply the rotation number of the boundary curve, as remarked above.  Thus, the move does not change the sum of relative Euler numbers.   A type (I) move occurs within a four-holed sphere $S'$; if the boundary of the sphere is given by oriented curves $a, b, c, d$ with $dcba = 1$, then it consists of replacing the decomposition along $da$ with a decomposition along 
$ab$.  It is easy to verify by the definition that, in either case, the sum of the Euler numbers of the two pants on $S'$ is given by 
$\rotild\left(\widetilde{\rho(a)}\right) + \rotild \left(\widetilde{\rho(b)}\right) + \rotild \left(\widetilde{\rho(c)}\right) + \rotild \left(\widetilde{\rho(d)}\right)$.
\end{proof}

\begin{proposition}[Additivity of Euler number] \label{prop:AdditiveEuler}
Let $\mathcal{P}$ be any decomposition of $\Sigma$ into pants.  Then $\eu(\rho) = \sum \limits_{P \in \mathcal{P}} \eu_P(\rho)$.  
\end{proposition}

By Lemma~\ref{lem:IndepPants}, we may use any pants decomposition to compute
the Euler class. By using a standard generating system $(a_1,\ldots,b_g)$
and cutting $\Sigma_g$ along geodesics freely homotopic to $a_i$,
$c_i=[a_i,b_i]$, for $i=1,\ldots,g$ and $d_i=c_i\cdots c_1$ for
$i=2,\ldots,g-1$, we recover the formula taken as a definition in
Proposition~\ref{prop:DefMilnor}.
 
We now return to our main goal: we prove that maximality of the Euler
class holds first on small subsurfaces, then globally on $\Sigma_g$.  
\begin{proposition}\label{prop:S4}
  Let $S\subset\Sigma_g$ be a subsurface homeomorphic to a four-holed
  sphere. Suppose that none of the boundary components of $S$ is separating
  in $\Sigma_g$, and let $\rho$ be a path-rigid, minimal representation
  mapping one boundary component of $S$ to a hyperbolic element of $\HOS$.
  Then, $\rho$ maps all four boundary components of $S$ to hyperbolic elements,
  and the relative Euler class $\eu_S(\rho)$ is equal to $\pm 2$.
\end{proposition}
In the statement above, we do not require that the boundary components are
geodesics for some metric on $\Sigma_g$, in particular, two of them may well
be freely homotopic.
\begin{proof}
  Put the base point inside of $S$. The complement $\Sigma_g\smallsetminus S$
  may have one or two connected components, since none of the curves
  of $\partial S$ are separating in $\Sigma_g$. In either case, we may find
  two based, nonseparating, simple closed curves $u,v\in\Gamma_g$, with
  $i(u,v)=0$, each having nonzero intersection number with exactly two of the
  boundary components of $S$, as shown in Figure~\ref{fig:S4}.
  Additionally, we may fix orientations for $u$ and $v$ and choose four
  elements $a,b,c,d\in\pi_1 S$, each
  freely homotopic to a different boundary component of $S$, with $dcba=1$,
  and such that
  $(a,u,d^{-1}a^{-1},v,d)$ and $(c,v,ad,u,b)$ are directed $5$-chains in
  $\Sigma_g$.  
  \begin{figure}[htb]
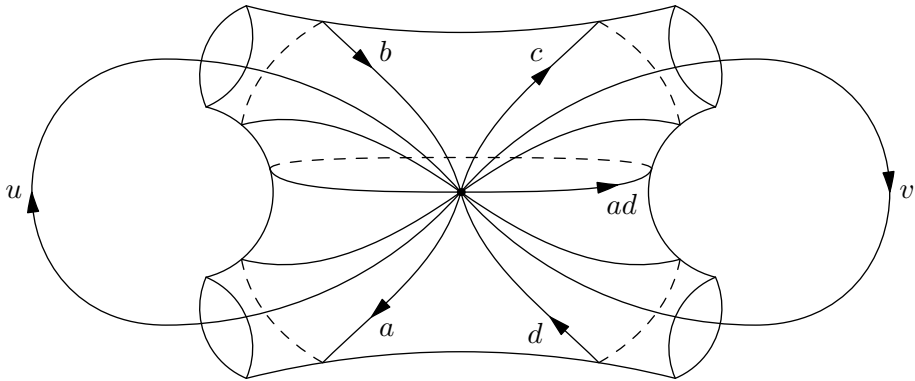

  \begin{asy}
    import geometry;
    
    point p1 = (-80,70), p2 = (-95,32);
    path contourhaut = p1..(0,60)..reflect((0,0),(0,1))*p1;
    path contourbas = reflect((0,0),(1,0))*contourhaut;
    path contourgauche = p2..(-70,0)..reflect((0,0),(1,0))*p2;
    path contourdroit = reflect((0,0),(0,1))*contourgauche;
    
    path u = (0,0)..(-110,-50){left}..(-110,50){right}..(0,0);
    path v = rotate(180)*u;
    path troudev = p1::((p1+p2)/2+10*dir(-60))::p2;
    path trouder = p1::((p1+p2)/2+10*dir(120))::p2;
    path b1 = (0,0){dir(145)}..point(contourgauche,0.34){dir(195)};
    path b3inv = (0,0){dir(105)}..point(contourhaut,0.36){dir(135)};
    path b2 = point(b1,1){dir(80)}..point(b3inv,1){dir(30)};
    path addev = point(contourgauche,0.8){dir(-75)}::(0,0){right}::point(contourdroit,0.8){dir(75)};
    path adder = point(addev,0){dir(75)}::(0,13){right}::point(addev,2){dir(-75)};
    
    draw(contourhaut); draw(contourbas); draw(contourdroit);
    draw(contourgauche); dot((0,0));
    draw(u,Arrow(Relative(0.5))); draw(v,Arrow(Relative(0.5)));
    label("\small $u$",point(u,1.5),W); label("\small $v$",point(v,1.5),E);
    draw(troudev); draw(trouder); draw(reflect((0,0),(0,1))*troudev);
    draw(reflect((0,0),(0,1))*trouder); draw(rotate(180)*troudev);
    draw(rotate(180)*trouder); draw(reflect((0,0),(1,0))*troudev);
    draw(reflect((0,0),(1,0))*trouder);
    draw(b1); draw(reverse(b3inv),Arrow(Relative(0.3)));
    label("\small $b$",point(b3inv,0.7),NE); draw(b2,dashed);
    draw(reflect((0,0),(1,0))*b1); draw(reflect((0,0),(1,0))*b2,dashed);
    draw(reflect((0,0),(1,0))*b3inv,Arrow(Relative(0.7)));
    label("\small $a$",point(reflect((0,0),(1,0))*b3inv,0.7),SE);
    draw(rotate(180)*b1); draw(rotate(180)*b2,dashed);
    draw(rotate(180)*reverse(b3inv),Arrow(Relative(0.3)));
    label("\small $d$",point(rotate(180)*b3inv,0.7),SW);
    draw(reflect((0,0),(0,1))*b1); draw(reflect((0,0),(0,1))*b2,dashed);
    draw(reflect((0,0),(0,1))*b3inv,Arrow(Relative(0.7)));
    label("\small $c$",point(reflect((0,0),(0,1))*b3inv,0.7),NW);
    draw(adder,dashed); draw(addev,Arrow(Relative(0.9)));
    label("\small $ad$",point(addev,1.7),S);
  \end{asy}
  \caption{A four-holed sphere and two $5$-chains}
  \label{fig:S4}
  \end{figure}
  As we have assumed that the image under $\rho$ of one of $a,b,c$ or $d$ is hyperbolic, Proposition~\ref{Presque} 
  implies that all the curves appearing in these
  $5$-chains are in fact hyperbolic.   

  Orient the circle so that $(u_-,(ad)_+,u_+,(ad)_-)$ are in cyclic order
  (as before, we drop the letter $\rho$ from the notation, for better
  readability).
  Then, Proposition~\ref{prop:Ordre5Chaines} applied to the two directed
  $5$ chains above gives the cyclic orderings
  \[ (a_-,u_-,a_+,(ad)_+,u_+,v_-,(ad)_-,d_-,v_+,d_+) \]
  and
  \[ (c_-,v_-,c_+,(ad)_-,v_+,u_-,(ad)_+,b_-,u_+,b_+). \]
  These two orderings together yield the cyclic ordering
  \[ ((ad)_-,d_-,d_+,a_-,a_+,(ad)_+,b_-,b_+,c_-,c_+). \]
  
  We now use this ordering to prove maximality of the Euler class.
  Let $\alpha,\beta,\gamma$ and $\delta$, respectively, denote the lifts
  of $\rho(a)$, $\rho(b)$, $\rho(c)$ and $\rho(d)$ to $\HOZ$ with translation
  number zero. Let $x=(ad)_-$ be the repelling fixed point of $ad$.

  Since $x$ has a repelling fixed point of $d$ immediately to the right, and
  an attracting fixed point of $d$ to the left, we have $\delta(x) < x$.
  By the same reasoning, if $y$ is any point in the interval between
  consecutive lifts of fixed points $a_+$ and $a_-$ containing $x$, then
  $\alpha(y) < y$. Since $ad(x) = x$, it follows that $\delta(x)$ must lie
  to the left of the lift of $a_+$, and we have $\alpha \delta (x) = x-1$.
  
  Since $cbad = 1$, we also have that $cb(x) = x$. Considering the location
  of repelling points of $b$ and $c$ and imitating the argument above,
  we have again $\beta(x) < x$, and also $\gamma \beta(x) < x$.
  It follows that $\gamma \beta(x) = x-1$, hence
  $\gamma\beta\alpha\delta(x)=x-2$, and $\eu_S(\rho)=-2$.
\end{proof}

Using this information about subsurfaces, we can prove that the Euler number
of $\rho$ is maximal.
\begin{proposition}
  Let $\rho$ be path-rigid, and suppose that $\rho$ admits a Fuchsian torus.
  Then $\rho$ has Euler number $\pm (2g- 2)$.
\end{proposition}

\begin{proof}
  After semi-conjugacy, we may assume that $\rho$ is minimal.
  Let $T(a,b)$ be a Fuchsian torus for $\rho$. By Corollary~\ref{Kt},
  we may suppose that $\rho(a)$ is hyperbolic.
  Ignoring the curve $b$, find a system of simple closed curves 
  $a_1 = a,  a_2, \ldots, a_{g-1}$, with each $a_i$ nonseparating, that decomposes $\Sigma_g$ into 
  a disjoint union of pairs of pants.  
  
  The dual graph of such a pants decomposition is connected
  (because $\Sigma_g$ is connected), so we may choose
  a finite path that visits all the vertices. In other words, we may choose
  a sequence $P_1,\ldots,P_N$ of pants from the decomposition
  (possibly with repetitions), that contains each of the pants of the
  decomposition, such that each two consecutive pants $P_i$ and $P_{i+1}$
  are distinct, but share a boundary component. Let $S_i$ denote the
  four-holed sphere obtained by taking the union of $P_i$ and $P_{i+1}$ along 
  a shared boundary curve.  (If $P_i$ and $P_{i+1}$ share more than one
  boundary component, choose only one).
  We may further assume that $a$ is one of the boundary curves of $S_1$.
  
  Starting with $S_1$ as the base case, and applying
  Proposition~\ref{prop:S4}, we inductively conclude that all boundary
  components of all the $S_i$ are hyperbolic,
  and that $\eu_{S_i}(\rho)=\pm 2$.  Thus, the
  contributions of $P_i$ and $P_{i+1}$ are equal, and equal to $\pm 1$,
  for all $i$. It follows that the contributions of all pairs of pants
  of the decomposition have equal contributions, equal to $\pm 1$.  By
  definition of the Euler class, we conclude that $\eu(\rho)=\pm (2g-2)$.
\end{proof}

The proof of Theorem \ref{thm:Main} now concludes by citing Matsumoto's
result of \cite{Matsumoto} that such a representation of maximal Euler number
is semi-conjugate to a Fuchsian representation.
\qed

\bibliographystyle{plain}

\bibliography{BabyBiblio}

\end{document}